\documentclass{amsart}

\usepackage{graphicx}
\usepackage{amssymb}
\usepackage[utf8]{inputenc}
\usepackage[T1]{fontenc}
\usepackage[british]{babel}
\usepackage[a4paper]{geometry}
\usepackage{amsmath,amssymb,amsthm}
\usepackage[shortlabels]{enumitem}
\usepackage[all]{xy}
\usepackage{tikz}
\usepackage{multicol}
\usepackage{float}
\usepackage{wrapfig}
\usepackage{comment} 

\usepackage{tikz-cd}

\usetikzlibrary{automata, arrows, arrows.meta, positioning}

\calclayout

\newcommand{\C}{\mathbb{C}}

\newcommand{\Sp}{\operatorname{Sp}}
\newcommand{\Mp}{\operatorname{Mp}}
\newcommand{\GSp}{\operatorname{GSp}}

\newcommand{\Char}{\operatorname{char}}

\newcommand{\gt}{\tilde{g}}
\newcommand{\htt}{\tilde{h}}

\newcommand{\ie}{{\it i.e.}\ }

\theoremstyle{plain}
\newtheorem{theo}{Theorem}[section]
\newtheorem{lem}[theo]{Lemma}
\newtheorem{cor}[theo]{Corollary}
\newtheorem{prop}[theo]{Proposition}
\theoremstyle{definition}
\newtheorem{defi}[theo]{Definition}

\newtheorem{rem}[theo]{Remark}

\theoremstyle{plain}
\newtheorem*{theo_sans_num}{Theorem}

\theoremstyle{definition}

\theoremstyle{plain}
\newtheorem{theorem}{Theorem}

\newtheorem{proposition}[theorem]{Proposition}

\title{The MVW involution of the metaplectic group}
\author{Shuichiro Takeda and Justin Trias}
\date{}

\begin{document}

\begin{abstract} The MVW involution -- named after Colette Moeglin, Marie-France Vignéras, and Jean-Loup Waldspurger -- is a fundamental dualizing involution in the representation theory of $p$-adic classical groups. It extends the well-known transpose–inverse automorphism for general linear groups. In this work, we establish the existence of the MVW involution for the metaplectic group over a non-archimedean local field $F$ of characteristic different from $2$ and with residue characteristic $p$. Our construction applies to representations over any coefficient field of characteristic distinct from $p$. \end{abstract}

\maketitle

\section*{Introduction}

Let $F$ be a non-archimedean local field of characteristic not $2$ and residue characteristic $p$. The main goal of this note is to provide a proof of the existence of the MVW involution for the metaplectic group. More precisely, let $(W,\langle \; , \; \rangle)$ be a nondegenerate symplectic space over $F$ and let $\Mp(W)$ be the metaplectic group, namely the unique non-trivial central extension of $\Sp(W)$ by $\mu_2$, so that we have
\[
1\longrightarrow\mu_2\longrightarrow \Mp(W)\longrightarrow \Sp(W)\longrightarrow 1.
\]
Let $\delta \in \textup{GSp}(W)$ be an element of similitude $-1$. The conjugation by $\delta$ on $\Sp(W)$ lifts uniquely to an automorphism of the group $\Mp(W)$, which we denote 
\[
\Mp(W)\longrightarrow\Mp(W),\quad \gt\mapsto {}^\delta\!\gt.
\]
Let $R$ be an algebraically closed field of characteristic $\ell \neq p$ and let $\textup{Irr}_R^\textup{gen}(\Mp(W))$ denote the isomorphism classes of smooth genuine irreducible representations of $\Mp(W)$. For $\pi \in \textup{Irr}_R^\textup{gen}(\Mp(W))$ we let $\pi^\vee$ be its contragredient and $\pi^\delta$ the precomposition with the automorphism defined by $\delta$, namely $\pi^{\delta}(\gt)=\pi({}^{\delta}\!\gt)$ for $\gt\in\Mp(W)$. The MVW involution asserts they are the same, as we prove in this paper:

\begin{theorem} \label{thm:main-theorem-R} For all $\pi \in \textup{Irr}_R^\textup{gen}(\Mp(W))$, we have $\pi^\delta \simeq \pi^\vee$. \end{theorem}

In particular, our main theorem applied to the case $R=\mathbb{C}$ is an improvement of the already existing result:

\begin{theo_sans_num}[{\cite[Chap 4, II.2]{mvw}}] Let $\pi \in \textup{Irr}_\mathbb{C}^\textup{gen}(\Mp(W))$. Suppose the trace-character of $\pi$ is given by a locally integrable function. Then $\pi^\delta \simeq \pi^\vee$. \end{theo_sans_num}

This type of theorem is usually proved by showing that the trace-characters of $\pi^\delta$ and $\pi^\vee$ are equal, which implies the theorem by the linear independence of trace-characters. When $F$ has characteristic zero and $G$ is a connected reductive group over $F$, trace-characters are easier to study due to Harish-Chandra's regularity theorem. This theorem ensures that the trace-character of an irreducible representation is given by a locally integrable function, supported on the set of regular semisimple elements of $G$. Therefore it is usually enough to check the equality on those semisimple elements solely. Wen Wei Li \cite[Th 4.3.2.]{wwl} proved that Harish-Chandra's regularity theorem holds for the metaplectic group, so this assumption in the theorem can be relaxed when $F$ is a $p$-adic field. When $F$ has positive characteristic, no analog of the regularity theorem is known and one usually has to consider other techniques from invariant distributions, so our main theorem is truly new in this case. Most of the dualizing involutions in positive characteristic -- such as the transpose-inverse for $\textup{GL}_n(F)$ or the MVW involution for classical groups (see \cite[Chap 4, II]{mvw} and \cite[Sec 13.4]{theta_book}) -- are obtained from two related methods. The first one is due to Gelfand and Kazhdan \cite{gk} and requires the action to be \textit{regular}. The second one is due to Bernstein and Zelevinski \cite{bz} and generalizes the results of Gelfand-Kazhdan to more general actions called \textit{constructive} actions.

From the point of view of complex representations, the second method encompasses the first one. However, from the point of view of $\ell$-modular representations, \textit{i.e.} replacing the coefficient field $\mathbb{C}$ by a field of positive characteristic $\ell \neq p$ such as $\overline{\mathbb{F}_\ell}$, it is no longer true because, typically, the second method does not supersede the first. Indeed, the first method puts more restrictions on the action but usually works for all characteristics $\ell$ \cite[Prop 2.5 \& Thm 3.7]{modular_theta}. The second one has opposite benefits in the sense that you can consider more general actions, at the price of excluding some values of $\ell$ \cite[Prop 2.4 \& Thm 3.6]{modular_theta}. Moreover, local integrability does not make sense in the modular setting. As a consequence of this fact, there is no modular analog of Harish-Chandra's regularity theorem at all, even when $F$ has characteristic zero, so the two methods we mentioned -- Gelfand-Kazhdan and Bernstein-Zelevinski -- are the only tools available to exhibit dualizing involutions. 

Regular and constructive actions can be rephrased using the language of finite stratifications; they correspond to finite stratifications where each quotient stratum is Hausdorff or locally Hausdorff, respectively. By translating properties of actions in these terms, we can highlight other types of actions that almost sit between regular and constructive, such as actions with finite regular Matryoshka\footnote{The term Matryoshka comes from the Russian doll that is nested in one another, and is due to the second author in \cite{modular_theta}.} strata (see Section \ref{def:finite-Matryoshka-regular-strata}). These actions seem very well-behaved for questions related to covering groups, which include the metaplectic group. For a covering group $1 \to A \to \widetilde{G} \to G \to 1$ where $A$ is a finite abelian group, these finite Matryoshka stratifications package in a nice way the pullback of a finite stratification on $G$ à la Gelfand-Kazhdan. This allows to extend their methods to covering groups provided the conjugation action on $G$ is regular, which is our main strategy in this paper. (We were not able to determine whether the conjugation action of the metaplectic group on itself is constructive.)

The following is an outline of this paper. In Section 1, we will define the notion of Matryoshka regular actions and explain how the existence of the MVW involution is reduced to the following key fact.
\begin{proposition}\label{P:conjugacy_assertion_in_Mp} For all $\tilde{g} \in \Mp(W)$, the elements ${}^{\delta}\!\tilde{g}$ and  $\tilde{g}^{-1}$ are conjugate in $\Mp(W)$. \end{proposition}
Our Section \ref{sec:conjugacy-classes-metaplectic-group} is devoted to a proof of this theorem. The basic idea is to lift the $\Sp(W)$-analog of this proposition by using the ``Jordan decomposition'' in $\Mp(W)$. Finally in Section 3, we will finish our proof.

\quad

Let us note that, besides the classical groups treated by Moeglin, Vign\'{e}ras and Waldspurger \cite{mvw} (\ie symplectic, orthogonal and unitary) in the complex setting, and by the second author \cite{modular_theta} in the $\ell$-modular setting, the groups for which the MVW-involution is known to exist (so far only in the complex case) include the following: the general Spin groups (and their variants) \cite{emory_takeda, Modal_Nadimpalli}, certain double covers of $\operatorname{GL}(2)$ \cite{BT}, and the similitude groups in odd residue characteristic \cite{roche_vinroot}. Also there is a simple proof for the general linear groups due to Tupan \cite{tupan}, whose method is empolyed in \cite{BT, roche_vinroot}. It should be also mentioned that, for quaternionic classical groups over $F$, no MVW involution exists, as shown by Lin-Sun-Tan \cite[Remarks (a)]{Lin-Sun-Tan}.

\quad

\section*{Notation and assumptions}
Throughout the paper, $F$ will be a non-archimedean local field of $\Char(F)\neq 2$ and residue characteristic $p$. We fix a nondegenerate symplectic space $W$ of $\dim_FW=2n$. We denote by $p:\Mp(W)\to\Sp(W)$ the canonical projection. (We believe that the reader will not confuse $p$ here with the residue characteristic of $F$.) For each $g\in Sp(W)$, we denote by $\gt$ an element in $\Mp(W)$ such that $p(\gt)=g$; certainly there are two choices for $\gt$, but we will use this notation whenever the ambiguity does not produce any confusion. 

Let $G$ be an arbitrary group. For each $g\in G$, we denote by $C_g$ the conjugacy class of $g$ in $G$.

\quad

\section*{Acknowledgments} 
The first author was partially supported by JSPS KAKENHI Grant Number 24K06648. The second author was partially supported by the EPSRC Grant EP/V061739/1. This paper was mostly completed while both authors were visiting the National University of Singapore in September 2025, whose hospitality we gratefully acknowledge. We also wish to thank Wee Teck Gan and Shaun Stevens for their stimulating advices and encouragement.

\section{Finite Matryoshka regular actions} \label{sec:actions-and-stratification}

In this section, we go over the notion of Matryoshka regular action, which was originally introduced in \cite{modular_theta} and apply it to our situation.

\subsection{Review of different actions} 
Let $X$ be a locally profinite topological space. We say an equivalence relation $\sim$ on $X$ is {\it regular} if $X\slash\!\sim$ equipped with the quotient topology is Hausdorff.

Let $G$ be a locally profinite group that is countable at infinity. Assume $G$ acts on $X$ continuously. We say that the action admits a {\it finite stratification} if there there exists a partition $X=\sqcup_{i \in I} X_i$ with $I$ finite such that
\begin{itemize}
\item $I$ is a totally ordered set;
\item each $X_i$ is $G$-stable;
\item $X_{\leq j} = \sqcup_{i \leq j} X_i$ is open for all $j \in I$. 
\end{itemize}
We call each $X_i$ a stratum of the finite stratification.

\begin{defi} We say the action of $G$ on $X$ admits a {\it finite regular stratification} if there exists a finite stratification $X=\sqcup_{i \in I}X_i$ such that each $X_i/G$ is Hausdorff. \end{defi}

By the results of \cite[Sec 2.1]{modular_theta}, a regular action in the sense of Gelfand-Kazhdan is an action which admits a finite regular stratification; also a constructive action in the sense of Bernstein-Zelevinski is an action which admits a finite stratification $X=\sqcup_{i \in I} X_i$ such that each $X_i\slash G$ is locally Hausdorff.

Given an equivalence relation $\sim$ on $X$, we call each equivalence class a $\sim$-packet. We then define

\begin{defi} \label{def:finite-Matryoshka-regular-strata} The action of $G$ on $X$ is {\it Matryoshka regular} if there exists a finite sequence of equivalence relations $(\sim_i)_{0, \dots, n}$ on $X$ such that
\begin{itemize}
\item $\sim_0$ is the trivial (least refined) relation (\text{i.e.} all elements are related), so there is only one $\sim_0$-packet, namely $X$;
\item $\sim_n$ is the $G$-action on $X$, so each $\sim_n$-packet is a $G$-orbit;
\item $\sim_{i+1}$ refines $\sim_i$; namely for $x, y\in X$, if $x\sim_{i+1} y$ then $x\sim_iy$;
\item the restriction of $\sim_{i+1}$ on each $\sim_i$-packet $\mathcal{P}_i$ is regular, so the quotient $\mathcal{P}_i / \hspace{-0.15cm} \sim_{i+1}$ is Hausdorff.
\end{itemize} 
Further, we say the action of $G$ on $X$ admits a finite Matryoshka regular stratification if it admits a finite stratification $X=\sqcup_{i \in I} X_i$ such that the action of $G$ on each $X_i$ is Matryoshka regular.
\end{defi}

If the action of $G$ on $X$ admits a finite Matryoshka regular stratification given by a sequence of equivalence classes $(\sim_i)_{i=0,\dots, n}$, then each $\sim_i$-packet is a union of some $\sim_{i+1}$-packets, so each packet has a nested structure like the Matryoshka doll.

It should be noted that a finite regular stratification is a special case of finite Matryoshka regular stratification with $n=1$. 

\subsection{Invariant distributions} Let $R$ be an algebraically closed field of characteristic $\ell \neq p$ and let $C_c^\infty(X,R)$ be the space of locally constant compactly supported $R$-valued functions on $X$. A distribution on $X$ is an $R$-linear form on $C_c^\infty(X,R)$. Homeomorphisms of $X$ naturally act on $C_c^\infty(X,R)$ as well as on the space of distributions. Our main theorem will be a consequence of: 

\begin{prop}\label{P:invariant-distributions-Matrysohka} 
Assume the action of $G$ on $X$ admits a finite Matryoshka regular stratification. Suppose that $\sigma : X \to X$ is a homeomorphism satisfying:
\begin{itemize}
\item if there exists a non-trivial $G$-invariant distribution $T$ on a packet $S$, then $\sigma(S)=S$ and $\sigma \cdot T = T$.
\end{itemize}
Then any $G$-invariant distribution on $C_c^\infty(X,R)$ is $\sigma$-invariant.
\end{prop}
\begin{proof}
    This is \cite[Prop 2.10] {modular_theta}
\end{proof}

\subsection{Covering groups} Let $G$ be a connected reductive group over $F$ and $A$ a finite abelian group. Let $\widetilde{G}$ be a central $A$-extension of $G$, so that we have
\[
1 \longrightarrow A \longrightarrow \widetilde{G} \xrightarrow{\;p\;} G\longrightarrow 1,
\]
where we denote by $p : \widetilde{G} \to G$ the projection map. In particular $\widetilde{G}$ is a locally profinite group that is countable at infinity. 

\begin{prop} If the conjugation action of $G$ on $G$ admits a finite regular stratification, then the conjugation action of $\widetilde{G}$ on $\widetilde{G}$ admits a finite Matryoshka regular stratification. \end{prop}

\begin{proof} Let $G=\sqcup_{i \in I} X_i$ be a finite regular stratification of $G$ for the conjugation action and let $\widetilde{X}_i:=p^{-1}(X_i)$ be the inverse image of $X_i$, so that
\[
\widetilde{G}=\bigsqcup_{i\in I}\widetilde{X}_i
\]
is a finite stratification of $\widetilde{G}$. We consider three equivalence relations $(\sim_0, \sim_1,\sim_2)$ on $\widetilde{G}$:
\begin{itemize}
\item $\tilde{g} \sim_0 \tilde{g}'$ for all $\tilde{g}, \tilde{g}'\in\widetilde{G}$;
\item $\tilde{g} \sim_1 \tilde{g}'$ if $p(\tilde{g})$ and $p(\tilde{g}')$ are $G$-conjugate;
\item $\tilde{g} \sim_2 \tilde{g}'$ if $\tilde{g}$ and $\tilde{g}'$ are $\widetilde{G}$-conjugate.
\end{itemize}
Note that $\sim_1$ and $\sim_2$ induce equivalence relations on each stratum $\widetilde{X}_i$, and $\sim_2$ refines $\sim_1$. 

In what follows, we show that $(\sim_0, \sim_1, \sim_2)$ makes the finite stratification $\sqcup_{i \in I} \widetilde{X}_i$ Matryoshka regular. First note that 
\[
\widetilde{X}_i\slash\!\sim_1\cong X_i/G,
\]
where the right-hand side is Hausdorff by our assumption. Let $\mathcal{P}_1\subseteq \widetilde{X}_i$ be a $\sim_1$-packet, so that
\[
\mathcal{P}_1=p^{-1}(C_g)
\]
for some $g\in G$, where $C_g$ is the conjugacy class of $g$ in $G$. Then we have to show the quotient
\[
\mathcal{P}_1\slash\!\sim_2=\mathcal{P}_1\slash\tilde{G}
\]
is Hausdorff. Note that the set $\mathcal{P}_1\slash\!\sim_2$ is finite because the finite group $A$ acts transitively on it. Hence by \cite[Prop 1.5]{bz} the set $\mathcal{P}_1\slash\!\sim_2$ contains an open point. Since the action of $A$ is homeomorphic, all the points are open and hence closed. Thus this quotient set is discrete and in particular Hausdorff.


\end{proof}

\begin{cor} \label{cor:conjugation-in-Mp-is_Matryoshka} The conjugation of $\Mp(W)$ on $\Mp(W)$ admits a finite Matryoshka regular stratification. \end{cor}
\begin{proof}
    Since the conjugation action of $\Sp(W)$ on $\Sp(W)$ is regular in the sense of Gelfand-Kazhdan \cite[Ex p. 103]{gk}, it admits a finite regular stratification. Hence the corollary follows from the above proposition. 
\end{proof}

We will apply Proposition \ref{P:invariant-distributions-Matrysohka} with $G=\Mp(W)$ acting on $X=\Mp(W)$ by conjugation and $\sigma: X\to X$ the ``conjugation by a similitude $-1$'' element.  

\section{Conjugacy classes in the metaplectic group} \label{sec:conjugacy-classes-metaplectic-group}

In this section, we prove Proposition \ref{P:conjugacy_assertion_in_Mp}, which is the key ingredient for our proof of the main theorem. For this purpose, we will first review several known facts about the metaplectic group. All of them are available in \cite{mvw} and/or \cite{theta_book}. We will then define the notion of Jordan decomposition for each element $\gt\in\Mp(W)$ and prove that $\gt$ has the Jordan decomposition if and only if $g$ does. By using the Jordan decomposition, we will finish the proof of Proposition \ref{P:conjugacy_assertion_in_Mp}.

\subsection{Some properties of $\Mp(W)$}

Let $\delta\in\GSp(W)$ be an element of similitude $-1$. For each $g\in\Sp(W)$, we write ${}^{\delta}\!g=\delta g\delta^{-1}$, which gives the group isomorphism $\Sp(W)\to\Sp(W), g\mapsto{}^{\delta}\! g$. It is known that there is a unique isomorphism 
\[
\Mp(W)\longrightarrow\Mp(W),\quad\tilde{g}\mapsto {}^{\delta}\!\tilde{g},
\]
such that $p({}^{\delta}\!\tilde{g})={}^\delta\! g$; namely the isomorphism $g\mapsto{}^\delta\! g$ uniquely lifts to the metaplectic group. (See \cite[Corollary 6.19]{theta_book}.) We often write
\[
{}^\delta\! g=\delta\gt\delta^{-1}.
\]
The notation is actually justified by \cite[Corollary 6.19]{theta_book}. Our goal here is to prove Proposition \ref{P:conjugacy_assertion_in_Mp}; namely ${}^{\delta}\!\tilde{g}$ and $\tilde{g}^{-1}$ are conjugate for all $\tilde{g} \in \Mp(W)$ and all $\delta\in\GSp(W)$ with similitude $-1$.

But certainly, it suffices to show this conjugacy assertion only for one $\delta$:
\begin{lem}
    Fix $\gt\in\Mp(W)$. For all $\delta, \delta'\in\GSp(W)$ both of which have similitude $-1$, we have $C_{{}^{\delta}\!\gt}=C_{{}^{\delta'}\!\gt}$.
\end{lem}
\begin{proof}
   Note that one can write $\delta'=\delta h$ for some $h\in\Sp(W)$. Hence one can see
   \[
   {}^{\delta'}\!\gt={}^{\delta}\!(\htt\gt\htt^{-1}),
   \]
   where $\htt\in\Mp(W)$ is any element such that $p(\htt)=h$. (See \cite[Corollary 6.19]{theta_book}.) Clearly $C_{\gt}=C_{\htt\gt\htt^{-1}}$.
\end{proof}

The analog of Proposition \ref{P:conjugacy_assertion_in_Mp} for $\Sp(W)$ is known to be true.
\begin{lem} \label{lem:conjugation_in_sp_delta} 
For all $g \in \Sp(W)$, we have $g^{-1} \in C_{\delta g \delta^{-1}}$. \end{lem}
\begin{proof}
    See \cite[Chap.\ 4, I.2, Prop]{mvw}.
\end{proof}

We will prove Proposition \ref{P:conjugacy_assertion_in_Mp} by ``lifting'' this result to $\Mp(W)$.

\quad

Let us also mention the following important fact we will frequently use.
\begin{lem}\label{L:two_elements_commute}
    Let $\gt, \htt\in \Mp(W)$. Then $\gt\htt=\htt\gt$ in $\Mp(W)$ if and only if $gh=hg$ in $\Sp(W)$; namely ``two elements commute in $\Mp(W)$ if and only if they commute in $\Sp(W)$''.
\end{lem}
\begin{proof}
    This is \cite[Chap.\ 2, II.5.\ Lemme]{mvw}. Also see \cite[Theorem 6.28]{theta_book} for a completely different proof.
\end{proof}

\subsection{Jordan decomposition}

In this subsection, we will introduce the notion of Jordan decomposition in $\Mp(W)$. For this purpose, let us first recall the usual notion of Jordan decomposition for reductive groups. Let $G$ be a reductive group over a field $F$. Then for each $g\in G(F)$, we say that $g$ has the Jordan decomposition if there exist a unique semisimple $s\in G(F)$ and unipotent $u\in G(F)$ such that
\[
g=su=us.
\]
It is well-known that if $\Char(F)=0$, then $g$ always has the Jordan decomposition. But in general, $g$ has the Jordan decomposition after extending scalars to a purely inseparable extension (possibly $F$ itself); namely for each $g\in G(F)$, there exist a purely inseparable extension $F_g$ (depending on $g$) and a unique semisimple $s\in G(F_g)$ and unipotent $u\in G(F_g)$ such that $g=su=us$. (See \cite[Chap I, 4.4]{borel_book}.)

Using the Jordan decomposition of $g\in\Sp(W)$ (provided it exists), one can define the notion of the Jordan decomposition of $\gt\in\Mp(W)$ as follows. Assume $g=su=us$ is the Jordan decomposition of $g$. The element $u$ is contained in a unipotent group $U$ which admits a splitting $U \to \textup{Mp}(W)$ by virtue of \cite[Appendix I]{mw} applied to the local case \cite[p.277]{mw}. Since $\Char (F)\neq 2$, the square map $U \to U$ is a bijection, so $U$ has no quadratic character and the previous splitting is unique. We denote it 
\[
U\longrightarrow \Mp(W),\quad v\mapsto v^*.
\]
Then there exists a unique lift $\tilde{s}\in\Mp(W)$ of $s$ such that
\[
\gt=\tilde{s}u^*,
\]
which we call the {\it Jordan decomposition} of $\gt$.

\subsection{Proof of Proposition \ref{P:conjugacy_assertion_in_Mp}}
Now, we are ready to prove Proposition \ref{P:conjugacy_assertion_in_Mp}. We will prove it in three stages: first for semisimple $\gt$, second for $\gt$ which admits the Jordan decomposition in $\Mp(W)$, and finally for general $\gt$.

The first one is already known:
\begin{lem}
    Let  $\gt\in\Mp(W)$ be such that $g$ is semisimple. Then Proposition \ref{P:conjugacy_assertion_in_Mp} holds.
\end{lem}
\begin{proof}
    See \cite[Proposition 13.38]{theta_book}.
\end{proof}

\begin{rem}
    If $\operatorname{char}(F)=0$ and $R=\C$, then this proposition is enough to prove our main theorem (Theorem \ref{thm:main-theorem-R}) thanks to Harish-Chandra's regularity theorem proven by Li \cite{wwl}. 
\end{rem}

Using this, we can prove the second one:
\begin{lem}\label{L:g_in_Sp_admits_Jordan_decomp}
Let $\gt\in\Mp(W)$ be such that $g$ admits the Jordan decomposition in $\Sp(W)$. Then Proposition \ref{P:conjugacy_assertion_in_Mp} holds.
\end{lem}
\begin{proof}
Let $g=su$ be the Jordan decomposition in $\Sp(W)$, so that we have the Jordan decomposition $\gt=\tilde{s}u^*$ in $\Mp(W)$. We know by Lemma \ref{lem:conjugation_in_sp_delta} that there exists $\gamma\in\GSp(W)$ of similitude $-1$ such that $\gamma g\gamma^{-1}=g^{-1}$. Hence we have $(\gamma s\gamma^{-1})(\gamma u\gamma^{-1})=s^{-1}u^{-1}$. By uniqueness of the Jordan decomposition, we have
\[
\gamma s\gamma^{-1}=s^{-1}\quad\text{and}\quad \gamma u\gamma^{-1}=u^{-1}.
\]

The semisimple part implies $\gamma \tilde{s}\gamma^{-1}=\epsilon\tilde{s}^{-1}$ for some $\epsilon\in\{\pm 1\}$. On the other hand, we already know from Lemma \ref{lem:conjugation_in_sp_delta} that there exists $\delta\in\GSp(W)$ with similitude $-1$ such that $\delta \tilde{s}\delta^{-1}=\tilde{s}^{-1}$. Hence we have
\[
(\delta^{-1}\gamma)\tilde{s}(\delta^{-1}\gamma)^{-1}=\epsilon\tilde{s}^{-1}.
\]
Noting $\delta^{-1}\gamma\in\Sp(W)$, we have $(\delta^{-1}\gamma)s(\delta^{-1}\gamma)^{-1}=s$ by projecting down to $\Sp(W)$. Namely $s$ and $\delta^{-1}\gamma$ are two commuting elements in $\Sp(W)$ and hence their preimages in $\Mp(W)$ commute by Lemma \ref{L:two_elements_commute}, which implies $\epsilon=1$. Thus we have
\[
\gamma \tilde{s}\gamma^{-1}=\tilde{s}^{-1}.
\]
As for the unipotent part, since the map $u\mapsto u^*$ is the unique splitting and $u \mapsto \gamma^{-1} (\gamma u \gamma^{-1})^* \gamma$ is a splitting too, we have $(\gamma u \gamma^{-1})^* = \gamma u^* \gamma^{-1}$ and
\[
(u^{-1})^*=(\gamma u\gamma^{-1})^*=\gamma u^*\gamma^{-1}.
\]

Hence we have
\[
\gamma \gt\gamma^{-1}=(\gamma\tilde{s}\gamma^{-1})(\gamma u^*\gamma^{-1})=\tilde{s}^{-1}(u^{-1})^*=(\tilde{s}u^*)^{-1},
\]
where for the last equality we used that the map $u\mapsto u^*$ is a homomorphism and that $\tilde{s}$ and $u^*$ commute because $s$ and $u$ commute. The lemma is proven.
\end{proof}

Now, we will take care of general $\gt$. For this case we will utilize the fact that $g$ always admits the Jordan decomposition in a purely inseparable extension.

For this purpose, let us introduce the following notation. Let $E/F$ be a finite (not necessarily separable) extension of $F$, and write
\[
\Sp_{2n}(F)=\Sp(W)\quad\text{and}\quad \Sp_{2n}(E)=\Sp(W\otimes_F E),
\]
so that we have the natural inclusion
\[
\Sp_{2n}(F)\subseteq\Sp_{2n}(E).
\]
We also write
\[
\Mp_{2n}(F)=\Mp(W)\quad\text{and}\quad \Mp_{2n}(E)=\Mp(W\otimes_FE).
\]
Note that we have the following situation
\[
\begin{tikzcd}
p^{-1}(\Sp_{2n}(F))\ar[d]\ar[r, phantom, "\scriptstyle\subseteq"]&\Mp_{2n}(E)\ar[d, "p"]\\
\Sp_{2n}(F)\ar[r, phantom, "\scriptstyle\subseteq"]&\Sp_{2n}(E)\rlap{\ ,}
\end{tikzcd}
\]
where $p:\Mp_{2n}(E)\to\Sp_{2n}(F)$ is the canonical projection. We then have the following.
\begin{prop}\quad
\begin{enumerate}[(a)]
    \item  If $[K:F]$ is even, then $p^{-1}(\Sp_{2n}(F))$ is the trivial cover; namely it is isomorphic to the direct product $\Sp_{2n}\times\mu_2$. 
    \item If $[K:F]$ is odd, then $p^{-1}(\Sp_{2n}(F))$ is the metaplectic group $\Mp_{2n}(F)$.
\end{enumerate}
\end{prop}
\begin{proof}
    This is \cite[Theorem 6.32]{theta_book}.
\end{proof}

Now, we are ready to finish our proof of Proposition \ref{P:conjugacy_assertion_in_Mp}.
\begin{proof}[Proof of Proposition \ref{P:conjugacy_assertion_in_Mp}]
Let $g\in\Sp_{2n}(F)$ be fixed. We know that there exists a purely inseparable extension $F_g/F$ such that $g$ admits the Jordan decomposition in $\Sp_{2n}(F_g)$. Since $F_g$ is purely inseparable over $F$ and $\Char(F)\neq 2$, we know $[F_g:F]$ is odd. By the above proposition, the preimage of $\Sp_{2n}(F)$ in $\Mp_{2n}(F_g)$ is precisely $\Mp_{2n}(F)$; namely we have the following commutative diagram
\[
\begin{tikzcd}
\Mp_{2n}(F)\ar[d]\ar[r, phantom, "\scriptstyle\subseteq"]&\Mp_{2n}(F_g)\ar[d]\\
\Sp_{2n}(F)\ar[r, phantom, "\scriptstyle\subseteq"]&\Sp_{2n}(F_g)\rlap{\ ,}
\end{tikzcd}
\]
where the vertical arrows are the canonical projections.

Now, by Lemma \ref{L:g_in_Sp_admits_Jordan_decomp}, we know that there exists $\gamma\in\GSp_{2n}(F_g)$ with similitude $-1$ such that $\gamma\gt\gamma=\gt^{-1}$. On the other hand, we know that there exists $\delta\in\GSp_{2n}(F)$ with similitude $-1$ such that $\delta g\delta^{-1}=g^{-1}$ by Lemma \ref{lem:conjugation_in_sp_delta}, which gives $\delta\gt\delta^{-1}=\epsilon \gt^{-1}$ for some $\epsilon\in\{\pm 1\}$. Hence inside $\Mp_{2n}(F_g)$ we have
\[
(\delta^{-1}\gamma)\gt(\delta^{-1}\gamma)^{-1}=\epsilon\gt,
\]
which gives $(\delta^{-1}\gamma)g(\delta^{-1}\gamma)^{-1}=g$ inside $\Sp_{2n}(F_g)$. Noting that $\delta^{-1}\gamma\in\Sp_{2n}(F_g)$, we know that $g$ and $\delta^{-1}\gamma$ are two commuting elements in $\Sp_{2n}(F_g)$, which implies that their preimages in $\Mp_{2n}(F_g)$ commute by Lemma \ref{L:two_elements_commute}, so $\epsilon=1$. Thus we have $\delta\gt\delta^{-1}=\gamma\gt\gamma^{-1}=\gt^{-1}$. Since $\delta\in\GSp_{2n}(F)$, the proposition is proven.
\end{proof}

\section{The MVW involution} \label{sec:mvw-involution}

Let $R$ be an algebraically closed field of characteristic $\ell \neq p$. We recall that a representation $(\pi,V)$ of a locally profinite group $G$ is smooth if, for all $v  \in V$, the stabilizer $\{ g \in G \ | \ \pi(g) \cdot v  = v\}$ is open. Equivalently, the map $G \times V \to V$ defining the representation is continuous if we endow $V$ with the discrete topology. We denote by $\textup{Irr}_R(G)$ the set of equivalence classes of smooth irreducible representations with coefficients in $R$. For the metaplectic group $\Mp(W)$, we say that a representation $(\pi, V)$ is genuine if the element $-1 \in \mu_2\subseteq \Mp(W)$ acts as $- \textup{id}_V$. We denote by $\textup{Irr}_R^\textup{gen}(\Mp(W))$ the set of equivalence classes of smooth irreducible genuine representations of $\Mp(W)$. Note that, when $\ell = 2$, a genuine representation actually factors through the symplectic group.

Now, we use Proposition \ref{P:invariant-distributions-Matrysohka} to show:

\begin{theo} For all $\delta \in \textup{GSp}(W)$ with similitude $-1$ and all $\pi \in \textup{Irr}_R(\Mp(W))$, we have $\pi^\delta \simeq \pi^\vee$. \end{theo}

\begin{proof} By choosing a Haar measure on $G=\Mp(W)$, we let $\mathcal{H}_R(G)$ be the Hecke algebra of the metaplectic group. We want to show that $\textup{tr}_{\pi^\vee}(f) = \textup{tr}_{\pi^\delta}(f)$ for all $f \in \mathcal{H}_R(G)$. Since the metaplectic group is unimodular, we have $\textup{tr}_{\pi^\vee}(f) = \textup{tr}_\pi(f^\vee)$ and $\textup{tr}_{\pi^\delta}(f) = \textup{tr}_\pi(f^\delta)$ where $f^\vee(g) = f(g^{-1})$ and $f^\delta(g) = f(\delta g\delta^{-1})$. Therefore the statement is equivalent to proving $\textup{tr}_\pi(f) = \textup{tr}_\pi(f^\sigma)$ where $\sigma(g) = (\delta g\delta^{-1})^{-1}$ \textit{i.e.} $\textup{tr}_\pi$ is $\sigma$-invariant.

We will use Proposition \ref{P:invariant-distributions-Matrysohka}. Thanks to Corollary \ref{cor:conjugation-in-Mp-is_Matryoshka} and Proposition \ref{P:conjugacy_assertion_in_Mp}, we only have to check that $\sigma$ preserves $G$-invariant distribution on each conjugacy class. Note that, since a single conjugacy class is a homogeneous space, this $G$-invariant distribution is necessarily a quotient Haar measure. But the morphism $\sigma^2$ is the conjugation by $\delta^2 \in \Sp(W)$, so its modulus satisfies $|\delta^2| = 1 = |\sigma|^2$ because $G$ is unimodular. The modulus also belongs to $p^\mathbb{Z}$, so we must have $|\sigma|=1$ \textit{i.e.} $\sigma$ preserves local $G$-invariant distributions. Therefore $\textup{tr}_\pi$ is $\sigma$-invariant and $\textup{tr}_{\pi^\vee} = \textup{tr}_{\pi^\delta}$. Since both $\pi^\vee$ and $\pi^\delta$ are irreducible, this equality of trace-characters implies the result by the linear independence of trace-characters for irreducible representations \cite[Th 1.1.3.7]{trias_thesis}. \end{proof}

\bibliographystyle{alpha}
\bibliography{refer}

\end{document}